\documentclass[reqno, 12pt]{amsart}
\usepackage{amsmath,amsxtra,amssymb,latexsym, amscd,amsthm}
\usepackage{graphicx,color}
\usepackage[utf8]{inputenc}
\usepackage[mathscr]{euscript}
\usepackage{mathrsfs}
\usepackage[english]{babel}
\usepackage{enumerate}
\usepackage{tikz}
\usepackage{tikz-cd}
\usepackage{hyperref}
\newtheorem {theorem}{Theorem}[section]
\newtheorem {corollary}[theorem]{Corollary}
\newtheorem {lemma}[theorem]{Lemma}
\newtheorem {proposition}[theorem]{Proposition}
\newtheorem {claim}[theorem]{Claim}

\theoremstyle{definition}

\newtheorem {definition}[theorem]{Definition}
\newtheorem {remark}[theorem]{Remark}

\newcommand{\ord}{{\rm ord\ }}

\newcommand{\sign}{{\rm sign }}
\newcommand{\sol}{{\rm sol }}
\def\ar{a\kern-.370em\raise.16ex\hbox{\char95\kern-0.53ex\char'47}\kern.05em}
\def\ees{{\accent"5E e}\kern-.385em\raise.2ex\hbox{\char'23}\kern-.08em}
\def\eex{{\accent"5E e}\kern-.470em\raise.3ex\hbox{\char'176}}
\def\AR{A\kern-.46em\raise.80ex\hbox{\char95\kern-0.53ex\char'47}\kern.13em}
\def\EES{{\accent"5E E}\kern-.5em\raise.8ex\hbox{\char'23 }}
\def\EEX{{\accent"5E E}\kern-.60em\raise.9ex\hbox{\char'176}\kern.1em}
\def\ow{o\kern-.42em\raise.82ex\hbox{
  \vrule width .12em height .0ex depth .075ex \kern-0.16em \char'56}\kern-.07em}
\def\OW{O\kern-.460em\raise1.36ex\hbox{
\vrule width .13em height .0ex depth .075ex \kern-0.16em \char'56}\kern-.07em}
\def\uw{u\kern-.42em\raise.82ex\hbox{
  \vrule width .12em height .0ex depth .075ex \kern-0.16em \char'56}\kern-.07em}
\def\UW{U\kern-.42em\raise1.55ex\hbox{
\vrule width .13em height .0ex depth .075ex \kern-0.16em \char'56}\kern-.07em}
\def\DD{D\kern-.7em\raise0.4ex\hbox{\char '55}\kern.33em}
\def\OOH{{\accent"5E O}\kern-.78em\raise.8ex\hbox{\char'22}\kern.28em}
\def\UWS{\' \UW }

\title{The numerical range of limits of real rational functions}

\author{S\~i Ti\d{\^e}p \DD inh$^1$}
\address{Institute of Mathematics, VAST, 18, Hoang Quoc Viet Road, Cau Giay District 10307, Hanoi, Vietnam}
\email{dstiep@math.ac.vn}

\author{Feng Guo$^2$}
\address{School of Mathematical Sciences, Dalian  University of Technology, Dalian, 116024, China}
\email{fguo@dlut.edu.cn}

\author{H\OOH NG \DD\UWS C NGUY\EEX N}
\address{TIMAS, Thang Long University, Nghiem Xuan Yem, Hanoi, Vietnam}
\address{Institute of Mathematics, VAST, 18, Hoang Quoc Viet Road, Cau Giay District 10307, Hanoi, Vietnam}
\email{nhduc82@gmail.com}

\author{TI\EES N-S\OW N PH\d{A}M$^4$}
\address{Department of Mathematics, Dalat University, 1 Phu Dong Thien Vuong, Dalat, Vietnam}
\email{sonpt@dlu.edu.vn}

\date{\today}

\begin{document}

\begin{abstract}
Given two nonzero polynomials $f, g \in \mathbb{R}[x_1,\dots,x_n]$ vanishing at the origin $0 \in \mathbb{R}^n$, we study the numerical range of the limit 
$$
\lim_{x \to 0} \frac{f(x)}{g(x)},
$$
denoted by $\mathcal{R}_{f,g}$, which is defined as the set of all possible limits of the rational function $\displaystyle\frac{f}{g}$ as $x$ tends to $0$.
When $n = 2$ and $g$ has a non-isolated zero at $0$, we describe all possible forms of $\mathcal{R}_{f,g}$ and provide an effective algorithm to compute this set based on Puiseux expansions.
This is a follow-up to our recent results~\cite{Dinh2025}, in which we only considered the case when $0$ is an isolated zero of $g$.
\end{abstract}
\maketitle

\section{Introduction}

Computing limits of multivariate functions is a fundamental problem in computational calculus. 
While the univariate case is well understood \cite{Gruntz1996,Salvy1999}, the multivariate setting, especially for rational functions, remains challenging and actively studied even in two variables \cite{Alvandi2016,Cadavid2013,Strzebonski2018,Strzebonski2021,Velez2017,Xiao2014,Xiao2015,Zeng2020}.

Several algorithms have been developed under the assumption that the denominator has an isolated zero at the given point. 
Notably, Cadavid, Molina and V{\'e}lez~\cite{Cadavid2013} introduced an approach (implemented in {\scshape Maple} as {\sf limit/multi}), later extended to higher dimensions \cite{Alvandi2016,Velez2017}, based on analyzing limits along the discriminant variety. 
Related criteria using real valuations and polynomial curves were obtained in \cite{Xiao2015}.

Without such assumptions, Xiao and Zeng~\cite{Xiao2014} reduced the problem to quantifier elimination, later refined using Sturm--type methods \cite{Zeng2020}. 
Alternative approaches based on cylindrical algebraic decomposition were proposed by Strzebo{\'n}ski \cite{Strzebonski2018,Strzebonski2021}, including reductions from analytic to rational cases when the denominator has an isolated zero.

In~\cite{Dinh2025}, we considered the problem of computing limits of real bivariate rational functions, in which, we provide necessary and sufficient conditions for the existence of an unique limit. 
This leads to an effective algorithm for deciding existence and computing the limit for polynomials with rational (or more generally computable) coefficients.
Furthermore, we also showed that if the denominator $g$ has an isolated zero at the point $0,$ then the numerical range $\mathcal{R}_{f, g}$, defined as the set of all possible limits along sequences approaching $0$,  is a closed interval in $\overline{\mathbb{R}}$ and can be explicitly determined. 
Our approach is geometric and based on Puiseux expansions, which is similar to Cadavid, Molina and V{\'e}lez~\cite{Cadavid2013}. 
More precisely, to verify the existence of the limit, it suffices to analyze locally near $0$ the behavior of $f$ along the half-branches of $g=0$, and of $f$ and $g$ along the half-branches of a polynomial curve.
However, instead of using the discriminant curve $G_{f,g}=0$ as in~\cite{Cadavid2013}, where
$$
G_{f,g} := y\!\left(g\frac{\partial f}{\partial x}-f\frac{\partial g}{\partial x}\right) - x\!\left(g\frac{\partial f}{\partial y}-f\frac{\partial g}{\partial y}\right).
$$
we used the critical curve $F_{f,g}=0$, where
$$
F_{f,g} := \frac{\partial f}{\partial x}\frac{\partial g}{\partial y} - \frac{\partial f}{\partial y}\frac{\partial g}{\partial x},
$$
which seems to be better than the discriminant curve since, in general, $\deg F_{f, g} < \deg G_{f, g}$ and the number of monomials of $F_{f, g}$ is smaller than that of $G_{f, g}.$
Note that we also provided in~\cite{Dinh2025-1} an algorithm for deciding the existence of the limit $\displaystyle\lim_{x \to 0} \frac{f(x)}{g(x)}$, based on computing the {\L}ojasiewicz exponent of a pair of real bivariate analytic functions. 
However, that method does not allow one to compute the numerical range.

In this paper, we consider two nonzero polynomials $f, g \in \mathbb{R}[x_1,\dots,x_n]$ such that $f(0)=g(0)=0$, and study the numerical range $\mathcal{R}_{f,g}$ when $0$ is a non-isolated zero of $g$. 
We prove that this set is a finite union of closed intervals. Moreover, under suitable sign conditions on $f$ over the zero set of $g$, the set $\mathcal{R}_{f,g}$ can be described explicitly (Theorem~\ref{general-range}).
When $n=2$, by means of Puiseux expansions, we show that $\mathcal{R}_{f,g}$ has one of the following forms:
$$
[-\infty,+\infty], \quad [-\infty,M^*], \quad [m^*,+\infty], \quad \text{or} \quad [-\infty,M^*]\cup[m^*,+\infty],
$$
where $m^*$ and $M^*$ belong to $\overline{\mathbb{R}}$ and are attained as limits of $\displaystyle\frac{f(x)}{g(x)}$ along the critical curve $F_{f,g}=0$ as $x \to 0$ (Theorem~\ref{2-range}).

The paper is organized as follows. In Section~\ref{SectionPreliminary}, we recall some notions and results from semi-algebraic geometry. 
Section~\ref{GeneralCase} studies the numerical range of rational functions in several variables. 
The two-dimensional case is treated in Section~\ref{2Case}. 
Finally, an effective algorithm for computing the numerical range is presented in Section~\ref{Algorithm}.

\section{Preliminaries} \label{SectionPreliminary}
Throughout this work we shall consider the Euclidean vector space ${\mathbb R}^n$ endowed with its canonical scalar product $\langle \cdot, \cdot \rangle,$ and we shall denote its associated norm by $\| \cdot \|.$  The closed ball (respectively, the open ball and the sphere) centered at ${x} \in \mathbb{R}^n$ of radius $r$ will be denoted by $\mathbb{B}_{r}({x})$ (respectively, $\mathring{\mathbb{B}}_{r}({x}),$ $\mathbb{S}_{r}({x})$). When ${x}$ is the origin $0$ of $\mathbb{R}^n$ we write $\mathbb{B}_{r}, \ \mathring{\mathbb{B}}_{r}$ and $\mathbb{S}_{r}$ instead of $\mathbb{B}_{r}(0),\ \mathring{\mathbb{B}}_{r}(0)$ and $\mathbb{S}_{r}(0)$, respectively.  

For a subset $X \subset \mathbb{R}^2$, denote by $\overline{X},\ \mathring{X}$ and $\partial X$ respectively the closure, the interior and the boundary of $X$; in particular, set $\overline{\mathbb R}:=\mathbb R\cup\{\pm\infty\}.$

The partial derivative of a differentiable function $f$ with respect to a variable $x_i$ is denoted by $\displaystyle\frac{\partial f}{\partial x_i}$ or $f_{x_i}$.

Let $\mathbb{C}\{x\}$ and $\mathbb{R}[x]$ denote, respectively, the {\em ring of convergent power series} and the {\em ring of real polynomials} in the variables $x$.

Now, we recall some notions and results of semi-algebraic geometry, which can be found in \cite{Benedetti1990, Bochnak1998, Hardt1980,HaHV2017, Dries1996}.

\begin{definition}{\rm
A subset $S$ of $\mathbb{R}^n$ is called {\em semi-algebraic} if it is a finite union of sets of the form
$$\{x \in \mathbb{R}^n \ | \  f_i(x) = 0, \ i = 1, \ldots, k;\ f_i(x) > 0, \ i = k + 1, \ldots, p\},$$
where all $f_{i}$ are polynomials.
In other words, $S$ is a union of finitely many sets, each defined by finitely many polynomial equalities and inequalities.
A function $f \colon S \rightarrow {\mathbb{R}}$ is said to be {\em semi-algebraic} if its graph
\begin{eqnarray*}
\{ (x, y) \in S \times \mathbb{R} \ | \ y =  f(x) \}
\end{eqnarray*}
is a semi-algebraic set.
}\end{definition}

A major fact concerning the class of semi-algebraic sets is its stability under linear projections.

\begin{theorem}[Tarski--Seidenberg theorem] \label{TarskiSeidenbergTheorem}
The image of any semi-algebraic set $S \subset \mathbb{R}^n$ under a projection to any linear subspace of $\mathbb{R}^n$ is a semi-algebraic set.
\end{theorem}

Furthermore, semi-algebraic maps and sets enjoy the following elementary properties:
\begin{itemize}
\item The closure, the interior and the boundary of a semi-algebraic set are semi-algebraic.

\item The image and the preimage of a semi-algebraic set under a semi-algebraic map are semi-algebraic.

\item The inverse of a semi-algebraic bijective map is semi-algebraic.

\item The composition of semi-algebraic maps is semi-algebraic.

\item The maximum and the minimum of a finite family of semi-algebraic functions are semi-algebraic.

\end{itemize}

The following well-known results will be of great importance for us.

\begin{theorem}[{Hardt's triviality}] \label{HardtTheorem}
Let $X$ and $Y$ be respectively semi-algebraic sets in $\mathbb{R}^n$ and $\mathbb{R}^m$, $f \colon X \rightarrow Y$ be a continuous semi-algebraic map.
Then there exists a partition of $Y$ into finitely many semi-algebraic subsets $Y_i, i = 1, \ldots, p,$ such that $f$ is semi-algebraically trivial over each $Y_i,$ i.e., there exist a semi-algebraic set $F_i \subset \mathbb{R}^n$ and a semi-algebraic homeomorphism
$h_i \colon f^{-1} (Y_i) \rightarrow Y_i \times F_i$ 
such that the following diagram commutes
$$\begin{tikzcd}
f^{-1}(Y_i) \arrow {r}{h_i}
\arrow{rd}{f}
& Y_i \times F_i \arrow {d}{\pi} \\
& Y_i
\end{tikzcd}$$
where $\pi$ is the projection on the first factor.
\end{theorem}

\begin{lemma}[curve selection]\label{CSL}
Let $X\subset \mathbb{R}^n$ be a semi-algebraic set.
Assume that there exists a sequence $z^k\in X$ such that $\displaystyle\lim_{k \to \infty} z^k  = z\in\overline X.$ 
Then there exists a $C^1$ semi-algebraic curve $\varphi \colon (0, \epsilon)\to X\setminus \{z\}$ such that $\displaystyle\lim_{t \to 0} \varphi(t) = z.$
 \end{lemma}

\begin{lemma} [monotonicity lemma] \label{MonotonicityLemma}
Let $f \colon (a, b) \rightarrow \mathbb{R}$ be a semi-algebraic function. Then there are finitely many points $a = t_0 < t_1 < \cdots < t_k = b$ such that the restriction of $f$ to each interval $(t_i, t_{i + 1})$ is analytic, and either constant, or strictly increasing or strictly decreasing.
\end{lemma}

\begin{lemma}[semi-algebraic choice] \label{LemmaSAChoice}
Let $S$ be a semi-algebraic subset of $\mathbb{R}^m \times \mathbb{R}^n.$ Denote by $\pi \colon \mathbb{R}^m \times \mathbb{R}^n \to \mathbb{R}^m$ the projection on the first $m$ coordinates. Then there is a semi-algebraic map $f \colon \pi(S) \to \mathbb{R}^n$ such that $(t, f(t)) \in S$ for all $t \in \pi(S).$ 
\end{lemma}

Now let us recall some facts on semi-algebraic stratification and dimension of semi-algebraic sets.
\begin{definition}{\rm 
A {\em semi-algebraic stratification} of a semi-algebraic subset $X$ of $\mathbb R^n$ is a finite partition $\mathcal S:=\{X_\alpha\}_{\alpha \in I}$ of $X$ such that:
\begin{enumerate}[\ $\bullet$]
\item Each $X_\alpha$ (called a {\em stratum of the stratification} or, briefly, a {\em stratum of $X$}) is a connected semi-algebraic analytic submanifold of $\mathbb R^n$.

\item The following {\em frontier condition} holds: If ${X_\alpha}\cap \overline{X}_\beta\ne\emptyset$, then ${X_\alpha}\subset \overline{X}_\beta.$
\end{enumerate} }
\end{definition}

It is well-known that every semi-algebraic set admits a semi-algebraic stratification; moreover, any finite family of semi-algebraic sets can be stratified simultaneously as follows (see, for example, \cite[Proposition~2.5.1]{Benedetti1990} and \cite[Proposition~9.1.8]{Bochnak1998}).

\begin{proposition}\label{CD1}
Let $X$ be a semi-algebraic subset of $\mathbb R^n$. For any finite family of semi-algebraic subsets $\{X_i\}_{i=1,\dots p}$ of $X$, there is a semi-algebraic stratification $\mathcal S$ of $X$ such that each $X_i$ is a union of strata in $\mathcal S$.
\end{proposition} 

Let $X\subset\mathbb R^n$ be a semi-algebraic set and let $\mathcal S:=\{X_\alpha\}_{\alpha \in I}$ be a semi-algebraic stratification of $X.$ We define the {\em dimension} of $X$ by
$$\dim X:=\max\{\dim X_\alpha:\ \alpha\in I\}.$$
It is not hard to check that this definition of dimension does not depend on the stratification of $X.$ For convenience, set $\dim\emptyset=-1$.  Given $x\in X$, the {\em dimension of $X$ at $x$} is defined by
$$\dim_x X:=\max\{\dim X_\alpha:\ \alpha\in I,\ x\in\overline{X}_\alpha\}.$$ 

Furthermore, we have the following properties of the dimension; see, for example, \cite[Section~2.8]{Bochnak1998}, \cite[Proposition~1.4]{HaHV2017}, \cite[Property~4.7]{Dries1996}.

\begin{lemma} \label{DimensionLemma}
Let $X$ be a semi-algebraic set in $\mathbb{R}^n.$ The following statements hold. 
\begin{enumerate}[{\rm (i)}]
\item If $X\ne\emptyset$ then $\dim(\overline{X}\setminus X)<\dim X.$ In particular, $\dim\overline{X}=\dim X.$
\item If $X$ is contained in a semi-algebraic set $Y \subset \mathbb{R}^n,$ then $\dim X \le \dim Y.$
\item If $f \colon X \to\mathbb{R}^m$ is a semi-algebraic mapping, then $\dim f(X)\leq\dim  X.$ 
Moreover, if $f$ is injective, then $\dim f(X)=\dim  X.$
\end{enumerate}
\end{lemma}

\begin{lemma}\label{CF}
Let $f,g \in \mathbb{R}[x_1,\dots,x_n]$ be two coprime polynomials. Then 
\[
\dim [f^{-1}(0) \cap g^{-1}(0)] < n-1.
\]
\end{lemma}

\begin{proof}
Let $A: = \mathbb{R}[x_1,\dots,x_n]$ and set
\[
X := \{x \in \mathbb{R}^n :\ f(x) = g(x) = 0\}.
\]
 If $X=\emptyset$, then it is clear that $\dim X<n-1$. Next, we assume that $X\neq\emptyset$.

As $X$ is a real algebraic set, by \cite[Proposition 2.8.2]{Bochnak1998}, we have 
$\dim X =  \dim A / \mathcal{I}(X)$, where $\mathcal{I}(X)$ denotes the vanishing ideal of $X$.
Since the ideal $\langle f, g\rangle\subset\mathbb{R}[x]$ generated by $f$ and $g$ is contained in 
$\mathcal{I}(X)$, it holds that 
\[
\dim X =  \dim A / \mathcal{I}(X)\le \dim A / \langle f, g\rangle.
\]
Thus, it remains to show that $\dim A/\langle f, g\rangle \le n-2$. 

Let $\mathfrak{p}$ be any minimal prime ideal containing $\langle f,g\rangle$ and
denote by $\mathrm{ht}(\mathfrak{p})$ the height of $\mathfrak{p}$.
We claim that $\mathrm{ht}(\mathfrak{p}) \ge 2$. 
Suppose for contradiction that $\mathrm{ht}(\mathfrak{p}) \leq 1$. 
Since $A$ is an integral domain and $\mathfrak p\neq(0)$, we have $\operatorname{ht}(\mathfrak p)\neq0$.
Thus  $\operatorname{ht}(\mathfrak p)=1.$
Since $A$ is a unique factorization domain, every height-one prime ideal in $A$ is principal. Thus, $\mathfrak{p} = \langle h\rangle$ for some irreducible polynomial $h \in A$. Since $f, g \in \langle f,g\rangle\subseteq \mathfrak{p} = \langle h\rangle$, it follows that $h$ is a common non-constant factor of $f$ and $g$, a contradiction.
Consequently, we have 
$$\dim A/\mathfrak{p} = n - \mathrm{ht}(\mathfrak{p}) \le n - 2.$$
Taking the maximum over all minimal prime ideals over $\langle f,g\rangle$, we conclude that
\[
\dim A/\langle f,g\rangle = \max_{\mathfrak{p} \supseteq \langle f,g\rangle} \dim A/\mathfrak{p} \le n - 2,
\]
which completes the proof.
\end{proof}


\begin{lemma}\label{Join}
Let $U$ and $U'$ be subsets in $\mathbb{R}^n$ with nonempty interiors.
Let $X \subset \mathbb{R}^n$ be a semi-algebraic set with $\dim(X) < n-1$.
Then for each $a \in U \setminus X$, there exists $a' \in U'$ such that
$$
[a,a'] \cap X = \emptyset,
$$
where $[a,a']$ denotes the closed segment joining $a$ and $a'$.
\end{lemma}

\begin{proof}
Since $\dim(X) < n$, the set $U \setminus X$ is nonempty. Fix $a \in U \setminus X$.
Suppose for contradiction that
\begin{equation}\label{Join1}
[a,z] \cap X \neq \emptyset \quad \text{for all } z \in U'.
\end{equation}
Fix a point $a' \in \mathring{U}'$ and choose $\delta > 0$ such that $\mathbb{B}_\delta(a') \subset U'$.
Define the $(n-1)$-dimensional disk
$$
D_\delta(a') := \mathbb{B}_\delta(a') \cap H,
$$
where $H$ is the hyperplane passing through $a'$ and perpendicular to the vector $a - a'$, i.e.,
$$
H := \{ z \in \mathbb{R}^n : \langle z - a', a - a' \rangle = 0 \}.
$$
Set
$$
S := \{(z,w) \in \mathbb{R}^n \times \mathbb{R}^n : z \in D_\delta(a'),\ w \in [a,z] \cap X \}.
$$
Let $\pi \colon \mathbb{R}^n \times \mathbb{R}^n \to \mathbb{R}^n$ be the projection onto the first $n$ coordinates.
By construction,
$$
D_\delta(a') \subset \mathbb{B}_\delta(a') \subset U',
$$
so by~\eqref{Join1}, we have $\pi(S) = D_\delta(a')$.
Since $S$ is semi-algebraic, by Lemma~\ref{LemmaSAChoice}, there exists a semi-algebraic map
$$
h \colon D_\delta(a') \to \mathbb{R}^n
$$
such that $(z,h(z)) \in S$ for all $z \in D_\delta(a')$.
For distinct points $z, v \in D_\delta(a')$, it is clear that
$h(z) \neq h(v)$ and hence $h$ is injective.
By Lemma~\ref{DimensionLemma}(iii),
$$
\dim\bigl(h(D_\delta(a'))\bigr) = \dim\bigl(D_\delta(a')\bigr) = n-1.
$$
Since $h(D_\delta(a')) \subset X$, Lemma~\ref{DimensionLemma}(ii) implies that
$$
\dim(X) \ge n-1,
$$
which contradicts the assumption $\dim(X) < n-1$.
\end{proof}

\section{Numerical ranges}\label{GeneralCase}

Let $f,g\in \mathbb{R}[x_1,\dots,x_n]$ and $X,Y$ be semi-algebraic sets with $z\in \overline X$. 

Set
$$
\sign(f,X)=
\left\{
\begin{array}{ll}
\{+\} & \text{if $f$ is positive on $X$},\\
\{-\}& \text{if $f$ is negative on $X$},\\
\end{array}
\right.
$$
If $X$ consists of a single point $z$, we write $\sign(f,z)$ instead of $\sign(f,X)$.

If each of the restrictions $f|_X$ and $g|_Y$ is either positive or negative, 
we say that $\sign(f,X)$ and $\sign(g,Y)$ {\em coincide} if 
$\sign(f,X)\cap \sign(g,Y)\ne \emptyset$; otherwise, we say that 
$\sign(f,X)$ and $\sign(g,Y)$ are {\em opposite}.

From now on, assume that $z=0$ for simplicity.

We say that $f$ {\em has positive $(${\rm respectively} negative$)$ sign on $X$ around $0$} and write $+\in\sign_0(f,X)$ (respectively $-\in\sign_0(f,X)$) if there exists a sequence $z^k\in X \setminus \{0\}$ tending to $0$ such that $f(z^k)>0$ (respectively $f(z^k)<0$) for all $k.$

The function $f$ {\em changes sign on $X$ around $0$} if $\{\pm\}\subset\sign_0(f,X);$ otherwise it {\em does not change sign on $X$ around $0$}.

Furthermore, write $f(X,0)\not\equiv 0$ if $\sign_0(f,X)\cap\{\pm\}\ne \emptyset$ and write $f(X,0)\equiv 0$ otherwise.
It is clear that if $f(X,0)\not\equiv 0$ and $f$ does not change sign around $0$, then either  $\sign_0(f,X)=\{+\}$ or $\sign_0(f,X)=\{-\}.$

The following corollary is a direct consequences of Lemma~\ref{CSL}.
\begin{corollary}\label{CSLA}
\begin{enumerate}[{\rm (i)}]
\item $f$ changes sign on $X$ around $0$ if and only if there exist $C^1$ semi-algebraic curves
$$
\varphi_i \colon (0,\epsilon)\to X\setminus\{0\}, \quad i=1,2,
$$
such that
$$
\|\varphi_i(t)\|=t
\quad \text{and} \quad
f(\varphi_1(t))\,f(\varphi_2(t))<0 \ \text{for all } t\in(0,\epsilon).
$$

\item $f(X,0)\not\equiv 0$ if and only if there exists a $C^1$ semi-algebraic curve
$$
\varphi \colon (0,\epsilon)\to X\setminus\{0\}
$$
such that
$$
\|\varphi_i(t)\|=t
\quad \text{and } \quad
f(\varphi(t))\ne0 \ \text{for all } t\in(0,\epsilon).
$$
\end{enumerate}
\end{corollary}

\begin{lemma}\label{By-Hardt}
For $r>0$ small enough, the following statements hold:
\begin{enumerate}[{\rm (a)}]
\item There is no critical point of $f$ and $g$ on respectively $\mathbb{B}_r\setminus f^{-1}(0)$ and $\mathbb{B}_r\setminus g^{-1}(0).$
\item The numbers of connected components of each of the following sets are stable for all $t\in (0,r]:\ \mathbb{S}_t \cap  g^{-1}(0),\ \mathbb{S}_t \setminus  g^{-1}(0), \ \mathbb{B}_t \cap  g^{-1}(0), \ \mathbb{B}_t \setminus  g^{-1}(0)$.

\noindent Let $\Omega$ be a connected component of $\mathbb{B}_r \setminus g^{-1}(0).$

\item The intersection $\Omega\cap \mathbb{S}_t$ is nonempty and connected for any $t\in (0,r].$
Moreover, $\overline\Omega\cap \mathbb{S}_t=\overline{\Omega\cap \mathbb{S}_t}.$
\item There exists $s_0>0$ such that $g^{-1}(s)\cap \Omega $ is nonempty and compact for all $s\in(0,s_0]$ whenever $g$ is positive on $\Omega$, and for all $s\in[-s_0,0)$ whenever $g$ is negative on $\Omega.$
\end{enumerate}
\end{lemma}
\begin{proof}
It suffices to show that, for each of Items~(a)--(c), there exists $r>0$ such that the corresponding statement holds. 
The desired result then follows by taking the minimum of the constants $r$ obtained.

(a) By Sard's theorem, the set of critical values of $f$ and $g$ is finite. 
In particular, there is $\delta > 0$ such that $(-\delta,0)\cup (0, \delta)$ does not contain a critical value of $f$ and $g.$ 
The statement follows by letting $r > 0$ be such that 
$$\max_{z\in \mathbb{B}_r}|f(z)|<\delta \quad \text{ and }\quad \max_{z\in \mathbb{B}_r}|g(z)|<\delta.$$

(b) Consider the function $\rho\colon \mathbb{R}^n\to \mathbb{R},\ x\mapsto \|x\|.$
By Theorem~\ref{HardtTheorem}, there exists $r>0$ such that the restrictions $\rho|_{g^{-1}(0)}$ and $\rho|_{\mathbb{R}^n \setminus g^{-1}(0)}$ are trivial over $(0,r].$
This immediately implies that the numbers of connected components of 
$\mathbb{S}_t \cap g^{-1}(0)$ and $\mathbb{S}_t \setminus g^{-1}(0)$ 
are stable for all $t \in (0,r]$.
By triviality, $\mathbb{B}_t \setminus g^{-1}(0)$ is homeomorphic to 
$[\mathbb{S}_t \setminus g^{-1}(0)] \times (0,t]$ for all $t \in (0,r]$. 
Hence the number of connected components of 
$\mathbb{B}_t \setminus g^{-1}(0)$ is also stable for all $t \in (0,r]$.

(c)  Let $r>0$ be given by Item~(b).
Since the restriction $\rho|_{\mathbb{R}^n \setminus g^{-1}(0)}$ is trivial over $(0,r]$,
it follows that the restriction of $\rho$ to each connected component of $\mathbb{B}_r \setminus g^{-1}(0)$ is also trivial over $(0,r]$.
In particular, $\rho|_{\Omega}$ is trivial over $(0,r]$.
Hence, there exist a semi-algebraic set $F$ and a semi-algebraic homeomorphism
$
h\colon \Omega \to (0,r] \times F 
$
such that the following diagram commutes:
\[
\begin{tikzcd}
\Omega \arrow{r}{h} \arrow{rd}[swap]{\rho}
&  (0,r] \times F  \arrow{d}{\pi} \\
& (0,r]
\end{tikzcd}
\]
Since $\Omega$ is nonempty and connected, it follows that  $\Omega \cap \mathbb{S}_t$ is also nonempty and connected for every $t\in(0,r]$.

Note that 
\[
\overline{\Omega}\cap \mathbb{S}_t \supset \overline{\Omega\cap \mathbb{S}_t}.
\]
Thus, to prove the equality it suffices to show that
\[
\overline{\Omega}\cap \mathbb{S}_t \subset \overline{\Omega\cap \mathbb{S}_t}.
\]
Let $z\in \overline{\Omega}\cap \mathbb{S}_t$. We show that 
$z\in \overline{\Omega\cap \mathbb{S}_t}$.
Choose a sequence $z^k\in \Omega$ such that $z^k\to z$.
Write
\[
h(z^k)=(\|z^k\|,u^k)
\ \text{ and }\
w^k:=h^{-1}(t,u^k).
\]
Since $\|z^k\|\to \|z\|=t$, we have
\[
\|(\|z^k\|,u^k)-(t,u^k)\|\to 0.
\]
By the continuity of $h^{-1}$, it follows that
\[
\|z^k-w^k\|\to 0,
\]
and hence $w^k\to z$.
Moreover, by construction, $w^k\in \Omega\cap \mathbb{S}_t$ for all $k$.
Therefore $z\in \overline{\Omega\cap \mathbb{S}_t}$.

(d) First of all, let us prove the nonemptiness. 
Consider the case $g$ is positive on $\Omega$, the remaining case can be treated similarly.
Let 
$$s_0=\frac{1}{2}\sup_{z\in\Omega}g(z)=\frac{1}{2}\max_{z\in \overline{\Omega}} g(z).$$
Hence $2s_0\in g(\overline\Omega)=\overline{g(\Omega)}.$
On the other hand, it is clear that $0\in \overline\Omega,$ so $0=g(0)\in g(\overline\Omega)=\overline{g(\Omega)}.$
Since $\Omega$ is connected, $g(\Omega)$ is also connected.
Therefore $(0,2s_0)\subset g(\Omega).$
In particular $(0,s_0]\subset g(\Omega),$ which implies that $g^{-1}(s)\cap \Omega $ is nonempty for all $s\in(0,s_0]$.

It remains to prove the compactness. 
Let $s \in (0,s_0]$. Since $\Omega \subset \mathbb{B}_r$, 
the set $g^{-1}(s)\cap \Omega$ is bounded. Hence it suffices to show that it is closed.
Let $\{z^k\}\subset g^{-1}(s)\cap \Omega$ be a sequence converging to a point $z$. 
We show that $z\in g^{-1}(s)\cap \Omega$.
Since $\mathbb{B}_r$ is closed and $z^k\in \mathbb{B}_r$ for all $k$, we have $z\in \mathbb{B}_r$. 
By the continuity of $g$,
$$
g(z)=\lim_{k\to\infty} g(z^k)=\lim_{k\to\infty} s=s.
$$
Since $s>0$, we have $z\notin g^{-1}(0)$. Consequently, $z$ lies in exactly one connected component 
of $\mathbb{B}_r\setminus g^{-1}(0)$.
By Lemma~\ref{CSL}, there exists a $C^1$ semi-algebraic curve 
$\varphi:(0,\epsilon)\to \Omega$ such that $\lim_{t\to 0}\varphi(t)=z$. 
Since $\varphi(0,\epsilon)\cup\{z\}$ is connected and $\varphi(0,\epsilon)\subset \Omega$, 
it follows that $z\in \Omega$.
Therefore $z\in g^{-1}(s)\cap \Omega$, showing that $g^{-1}(s)\cap \Omega$ is closed and hence compact.
\end{proof}

Let $f,g\in\mathbb{R}[x_1,\dots,x_n]$ with $f(0)=g(0)=0$ and $g\not\equiv 0$. 
Assume that $X$ is a subset of $\mathbb{R}^n$ with $0\in \overline X.$
\begin{definition}\label{DefRange}
The {\em numerical range} of limits on $X$ at $0$ of the rational function $\frac{f}{g}$ is defined to be the following set:
$$\mathcal{R}_{f, g}^X :=\left\{
\begin{array}{lll}
L \in \overline{\mathbb{R}} & : & \text{there is a sequence } z^k\in X \setminus g^{-1}(0)  \text{ such that } \\
&&z^k\to0 \ \text{ and }\ \displaystyle\lim_{k\to \infty}\frac{f(z^k)}{g(z^k)}=L 
\end{array}\right\}.
$$
If $X=\mathbb{R}^n,$ we write $\mathcal{R}_{f, g}$ instead of $\mathcal{R}_{f, g}^{\mathbb{R}^n}$ and call this set the {\em numerical range} of limits at $0$ of the rational function $\frac{f}{g}$.
\end{definition}

Fix an $r>0$ given by Lemma~\ref{By-Hardt} and let $\Omega=\Omega_r$ be a connected component of $\mathbb B_r\setminus g^{-1}(0).$ 

\begin{remark}
\begin{itemize}
\item Either $\sign(g,\Omega)=\{+\}$ or $\sign(g,\Omega)=\{-\}.$
\item If $Y \subset X$ then $\mathcal{R}_{f, g}^Y \subset \mathcal{R}_{f, g}^X.$
\item We have 
$$\mathcal{R}_{f, g}=\mathcal{R}_{f, g}^{\mathbb{B}_r}=\displaystyle\bigcup_\Omega \mathcal{R}_{f, g}^\Omega,$$ 
where the union is taken over all connected components of $\mathbb B_r\setminus g^{-1}(0).$
Hence, to study $\mathcal{R}_{f, g}$, it suffices to study each $\mathcal{R}_{f, g}^\Omega.$
\item It is possible that the intersection of a connected component of $\mathbb{R}^n \setminus g^{-1}(0)$ with $\mathbb{B}_r$ is not connected for $r>0$ arbitrarily small. 
Therefore, the family of set $\{\mathcal{R}_{f, g}^\Omega\}$ depends on $r.$ 
Nevertheless, this family is constant when $r>0$ is sufficiently small.
\end{itemize}
\end{remark}

The following proposition shows that, when computing the numerical range of rational functions, it suffices to consider the case where the numerator and the denominator are coprime.

\begin{proposition}\label{FR}
Let $f,g\in\mathbb{R}[x_1,\dots,x_n]$ with $g\not\equiv 0$ and assume that $f=hf_1$ and $g=hg_1$, where $f_1,g_1,h\in \mathbb{R}[x_1,\dots,x_n]$.
Then 
$$\mathcal{R}_{f,g}=\mathcal{R}_{f_1,g_1}.$$
\end{proposition}

\begin{proof}
Let $L\in \mathcal{R}_{f,g}$. 
Then there exists a sequence $z^k\to 0$ such that $ g(z^k)\ne 0$ for all $k$ and
$$
\lim_{k\to\infty}\frac{f(z^k)}{g(z^k)}=L.
$$
Since $g(z^k)\neq 0$, it follows that $h(z^k)\neq 0$ and $g_1(z^k)\neq 0$.
Hence
$$
\lim_{k\to\infty}\frac{f_1(z^k)}{g_1(z^k)}
=
\lim_{k\to\infty}\frac{h(z^k)f_1(z^k)}{h(z^k)g_1(z^k)}
=
\lim_{k\to\infty}\frac{f(z^k)}{g(z^k)}
= L.
$$
Thus $L\in \mathcal{R}_{f_1,g_1}$, and therefore
$$
\mathcal{R}_{f,g} \subset \mathcal{R}_{f_1,g_1}.
$$

Conversely, let $L_1\in \mathcal{R}_{f_1,g_1}$ and let $z^k\to 0$ be a sequence such that $g_1(z^k)\ne 0$ for all $k$ and
$$
\lim_{k\to\infty}\frac{f_1(z^k)}{g_1(z^k)}=L_1.
$$
Note that $h\not\equiv 0$ and since, for each $k$, the function $\frac{f_1}{g_1}$ is continuous at $z^k$, it follows that for each $k$ there exists a point $w^k\in \mathbb{R}^n$ such that
$$
\|w^k-z^k\|<\frac{1}{k}, \quad 
h(w^k)\ne 0, \quad 
g_1(w^k)\ne 0, \quad
\left|\frac{f_1(w^k)}{g_1(w^k)}-\frac{f_1(z^k)}{g_1(z^k)}\right|<\frac{1}{k}.
$$
Hence
$$
\begin{array}{lll}
\displaystyle\left|\frac{f(w^k)}{g(w^k)}-L_1\right|
&=&\displaystyle\left|\frac{h(w^k)f_1(w^k)}{h(w^k)g_1(w^k)}-L_1\right|\\
&=&\displaystyle\left|\frac{f_1(w^k)}{g_1(w^k)}-L_1\right|\\
&\leqslant&\displaystyle\left|\frac{f_1(w^k)}{g_1(w^k)}-\frac{f_1(z^k)}{g_1(z^k)}\right|
+\left|\frac{f_1(z^k)}{g_1(z^k)}-L_1\right|\to 0.
\end{array}
$$
Therefore $L_1\in \mathcal{R}_{f,g}$.
\end{proof}

\begin{theorem}\label{general-range}
Let $f,g\in\mathbb{R}[x_1,\dots,x_n]$ with $f(0)=g(0)=0$ and $g\not\equiv 0$. 
The following statements hold:
\begin{enumerate}[{\rm (i)}]
\item The set $\mathcal{R}_{f,g}^{\Omega}$ is a closed interval. 
Moreover, it is unbounded if $f(g^{-1}(0)\cap \overline{\Omega},0)\not\equiv 0$. 
More precisely, if 
$$
\sign_0(f,g^{-1}(0)\cap \overline{\Omega})\cap \sign(g,\Omega)\neq \emptyset,
$$
then $+\infty \in \mathcal{R}_{f,g}^{\Omega}$; and $-\infty \in \mathcal{R}_{f,g}^{\Omega}$ if 
$$
\sign_0(f,g^{-1}(0)\cap \overline{\Omega})\cup \sign(g,\Omega)=\{\pm\}.
$$

\item  If $f$ changes sign on $g^{-1}(0)\cap\overline{\Omega}$ around $0$, then $\mathcal{R}_{f, g}^\Omega=\overline{\mathbb R}.$
\item Assume that $f^{-1}(0)\cap g^{-1}(0)\cap\overline{\Omega}=\{0\}$ and that $f$ does not change sign on $g^{-1}(0)\cap\overline{\Omega}$ around $0.$
Set 
$$\frak{m}=\frak{m}(\Omega):=\displaystyle \lim_{s\to 0}\frac{1}{s}\min_{z\in g^{-1}(s)\cap\Omega}f(z) \ \text{ and }\ \frak{M}=\frak{M}(\Omega):=\displaystyle \lim_{s\to 0}\frac{1}{s}\max_{z\in g^{-1}(s)\cap\Omega}f(z).$$
Then 
$$\mathcal{R}_{f, g}^\Omega=
\begin{cases}
[\frak{m},+\infty] & \text{if $\sign(f,(g^{-1}(0)\cap \overline{\Omega})\setminus\{0\})= \sign(g,\Omega)=\{+\}$} \\
[\frak{M},+\infty] & \text{if $\sign(f,(g^{-1}(0)\cap \overline{\Omega})\setminus\{0\})=\sign(g,\Omega)=\{-\}$}\\
[-\infty,\frak{M}] & \text{if $\sign(f,(g^{-1}(0)\cap \overline{\Omega})\setminus\{0\})=\{-\}$ and $\sign(g,\Omega)=\{+\}$} \\
[-\infty,\frak{m}] & \text{if $\sign(f,(g^{-1}(0)\cap \overline{\Omega})\setminus\{0\})=\{+\}$ and $\sign(g,\Omega)=\{-\}$.}
\end{cases}
$$
\item Assume that $f$ and $g$ are coprime in $\mathbb{R}[x_1,\dots,x_n]$. 
If $g$ does not change sign on $\mathbb{R}^n$ around $0,$ then $\mathcal{R}_{f, g}$ is connected.
\end{enumerate}
\end{theorem}

Before proving Theorem~\ref{general-range}, we need the following lemma.

\begin{lemma}\label{+-St}
Let $t>0.$
If $f|_{g^{-1}(0)\cap\overline{ \Omega\cap \mathbb{S}_t}}\not\equiv 0,$ then $\frac{f}{g}(\Omega\cap \mathbb{S}_t)$ is unbounded.
Furthermore,
\[
+\infty \in \overline{\frac{f}{g}(\Omega \cap \mathbb{S}_t)}
\]
whenever \(\operatorname{sign}(f,z)\) and \(\operatorname{sign}(g,\Omega)\) coincide for some 
$z \in g^{-1}(0)\cap\overline{\Omega \cap \mathbb{S}_t}$; 
and
\[
-\infty \in \overline{\frac{f}{g}(\Omega \cap \mathbb{S}_t)}
\]
whenever they are opposite.
\end{lemma}

\begin{proof}
Let $z\in g^{-1}(0)\cap\overline{ \Omega\cap \mathbb{S}_t}$ be such that $f(z)\ne 0$.
Let $z^k\in\Omega\cap \mathbb{S}_t$ be a sequence tending to $z.$
By continuity, $f(z^k)\to f(z)\ne 0$ and $g(z^k)\to g(z)=0,$ so $\displaystyle\frac{f(z^k)}{g(z^k)}\to\infty,$ which implies that $\frac{f}{g}(\Omega\cap \mathbb{S}_t)$ is unbounded.

Observe that, if \(\operatorname{sign}(f,z)\) and \(\operatorname{sign}(g,\Omega)\) coincide,  then for $k$ large enough, $f(z^k)$ and $g(z^k)$ have the same sign, so $+\infty \in \overline{\frac{f}{g}(\Omega \cap \mathbb{S}_t)}$.
On the other hand, if $\sign(f,z)$ and $\sign(g,\Omega)$ are opposite, then for $k$ large enough, $f(z^k)$ and $g(z^k)$ have the opposite signs, so $-\infty \in \overline{\frac{f}{g}(\Omega \cap \mathbb{S}_t)}$.

\end{proof}

\begin{proof}[Proof of Theorem~\ref{general-range}]
(i) Let $r$ be given by Lemma~\ref{By-Hardt} and consider the set
$$X=\left\{(z,t,w)\in \mathbb{R}^{n}\times \mathbb{R}\times \mathbb{R}:\ z\in\Omega,\ t=\|z\|,\ w=\frac{f}{g}(z)\right\}.$$
Observe that
$$\mathcal{R}_{f, g}^\Omega=\pi\left(\overline{X}\cap \{t=0\}\right),$$
where $\pi(z,t,w)=w$ is the projection on the last coordinate. 
It is clear that $\mathcal{R}_{f, g}^\Omega$ is closed and semi-algebraic so to prove that $\mathcal{R}_{f, g}^\Omega$ is a closed interval, it suffices to show that it is connected.
Suppose that $\mathcal{R}_{f, g}^\Omega$ is not connected.
Then it is not hard to see that the non-empty open semi-algebraic set $\mathbb{R} \setminus \mathcal{R}_{f, g}^\Omega$ has at least one connected component being a bounded open interval, denoted by $(a,b).$
Since $a$ and $b$ is contained in $\mathcal{R}_{f, g}^\Omega$, by Lemma~\ref{CSL}, there exists $C^1$ semi-algebraic curves $\varphi_i\colon (0,\epsilon )\to\Omega,\ i=1,2,$ such that 
\begin{equation}\label{GM1}
\lim_{t\to 0}\varphi_i(t)=0, \ \  \lim_{t\to 0}\frac{f}{g}(\varphi_1(t))=a \ \text{ and }\ \lim_{t\to 0}\frac{f}{g}(\varphi_2(t))=b.
\end{equation}
In addition, reparametrizing $\varphi_i$ if necessary, we may assume that $\|\varphi_i(t)\|=t.$
Hence 
\begin{equation}\label{GM2}
\varphi_i(t)\in \Omega\cap \mathbb{S}_t.
\end{equation}
By~\eqref{GM1}, for $t>0$ small enough, we have 
\begin{equation}\label{GM3}
\frac{f}{g}(\varphi_1(t))<\frac{a+b}{2} \ \text{ and }\ \frac{f}{g}(\varphi_2(t))>\frac{a+b}{2}.
\end{equation}
On the other hand, by Lemma~\ref{By-Hardt}(c), for each $t\in(0,r],$ the set $\Omega\cap \mathbb{S}_t$ is connected. 
So,  by continuity, $\frac{f}{g}(\Omega\cap \mathbb{S}_t)$ is also connected.
Combining this with~\eqref{GM2} and~\eqref{GM3} yields 
$$\frac{a+b}{2}\in \frac{f}{g}(\Omega\cap \mathbb{S}_t).$$
Letting $t\to 0$, we get
$$\frac{a+b}{2}\in \mathcal{R}_{f, g}^\Omega,$$
which is a contradiction, since $(a,b) \subset \mathbb{R} \setminus \mathcal{R}_{f, g}^\Omega.$
Consequently, $\mathcal{R}_{f, g}^\Omega$ is connected and so is a closed interval.

Now assume that $f(g^{-1}(0)\cap\overline{\Omega},0)\not\equiv 0.$
We will show that \( \mathcal{R}_{f,g}^{\Omega} \) is unbounded.
By Corollary~\ref{CSLA}(ii), there exists a $C^1$ semi-algebraic curve
\begin{equation}\label{CSLA1}
\varphi \colon (0,\epsilon)\to [g^{-1}(0)\cap\overline{\Omega}]\setminus\{0\}
\end{equation}
such that
$$
\|\varphi(t)\|=t
\quad \text{and } \quad f(\varphi(t))\ne 0 \ \text{for all } t \in (0,\epsilon).
$$
Therefore,
\[
f|_{g^{-1}(0)\cap\overline{\Omega} \cap \mathbb{S}_t} \not\equiv 0 
\quad \text{for all } t \in (0,\epsilon).
\]
Moreover, $\overline{\Omega} \cap \mathbb{S}_t =\overline{\Omega \cap \mathbb{S}_t} $ by  Lemma~\ref{By-Hardt}(c).
So in view of Lemma~\ref{+-St},  $\frac{f}{g}(\Omega \cap \mathbb{S}_t)$ is unbounded. 

Without loss of generality, we may assume that $g$ is positive on $\Omega.$
Assume first that 
$$
\sign_0(f,g^{-1}(0)\cap \overline{\Omega})\cap \sign(g,\Omega)\neq \emptyset.
$$
Then the curve $\varphi$ given by~\eqref{CSLA1} can be chosen such that $f(\varphi(t))>0$ for all $t \in (0,\epsilon)$.
Hence $\sign(f,\varphi(t))$ and $\sign(g,\Omega)$ coincide, and therefore $\displaystyle+\infty \in \overline{\frac{f}{g}(\Omega \cap \mathbb{S}_t)}$ for all $t \in (0,\epsilon)$.
Letting $t\to 0,$ we obtain $+\infty\in \mathcal{R}_{f,g}^{\Omega}.$

On the other hand, if 
$$
\sign_0(f,g^{-1}(0)\cap \overline{\Omega})\cap \sign(g,\Omega)=\{\pm\},
$$
then the curve $\varphi$ given by~\eqref{CSLA1} can be chosen such that $f(\varphi(t))<0$ for all $t \in (0,\epsilon)$.
So $\sign(f,\varphi(t))$ and $\sign(g,\Omega)$ are opposite, and hence $\displaystyle-\infty \in \overline{\frac{f}{g}(\Omega \cap \mathbb{S}_t)}$ for all $t \in (0,\epsilon)$.
Letting $t\to 0,$ we get $-\infty\in \mathcal{R}_{f,g}^{\Omega}.$

(ii) By Item~(i), $\mathcal{R}_{f,g}^{\Omega}$ is connected. 
On the other hand, by assumption, $\{\pm\}\subset \sign_0(f,g^{-1}(0)\cap\overline{\Omega})$. 
Hence, by Item~(i), we have $\{\pm\infty\}\subset \mathcal{R}_{f,g}^{\Omega}$. 
The statement follows.

(iii)
Consider the case when $\sign(f,(g^{-1}(0)\cap \overline{\Omega})\setminus\{0\})= \sign(g,\Omega)=\{+\}$, the remaining cases can be treated similarly.
We need the following claim.
\begin{claim}\label{2Cases}
Let $s_0>0$ be such that $g^{-1}(s)\cap \Omega$ is nonempty and compact for all $s\in(0,s_0),$ in view of Lemma~\ref{By-Hardt}(d).
Denote by $\sol(s) \subset g^{-1}(s)\cap \Omega$  the set of optimal solutions of the problem 
$$\min_{z\in g^{-1}(s)\cap \Omega}f(z).$$
Let $\alpha\colon (0,s_0)\to \Omega$ be a semi-algebraic arc such that $\alpha(s)\in\sol(s)$, given by Lemma~\ref{LemmaSAChoice}, chosen in such a way that 
$\displaystyle\alpha_0:=\lim_{s\to 0}\alpha(s)=0$ if $0\in \overline{\{\sol(s):\ s\in (0,s_0)\}}$.
Then we have
$$
\frak{m}=\displaystyle\lim_{s\to 0}\frac{f(\alpha(s))}{g(\alpha(s))} \ \text{ and  }\ \frak{m},+\infty\in\mathcal{R}_{f,g}^{\Omega}.
$$
Moreover, if $\alpha_0\ne0$ then $\frak{m}=+\infty.$
\end{claim}
\begin{proof}
The equality 
$\frak{m}=\displaystyle\lim_{s\to 0}\frac{f(\alpha(s))}{g(\alpha(s))}$ is straightforward.
By assumption and by Item~(i), $+\infty \in \mathcal{R}_{f, g}^\Omega.$ 
Moreover, if $\alpha_0=0,$ then it is also clear that $\frak{m}\in\mathcal{R}_{f,g}^{\Omega}.$
So it remains to consider the case when $\alpha_0\ne 0.$
Note that 
$$\displaystyle g(\alpha_0)=\lim_{s\to 0}g(\alpha(s))=\lim_{s\to 0}s=0.$$
Thus $\alpha_0\in (g^{-1}(0)\cap\overline{\Omega})\setminus\{0\}$ and so by assumption  $f^{-1}(0)\cap g^{-1}(0)\cap\overline{\Omega}=\{0\}$, we get $f(\alpha_0) > 0.$
Combining this with the fact that $g(\alpha(s))>0$ for any $s\in (0,s_0)$ yields 
$$
\displaystyle\lim_{s\to 0}\frac{f(\alpha(s))}{g(\alpha(s))}=+\infty.
$$
Hence $\frak{m}=+\infty.$
Therefore $\frak{m}\in \mathcal{R}_{f, g}^\Omega.$
\end{proof}

By Item~(i), $\mathcal{R}_{f, g}^\Omega$ is connected.
Moreover, in view of Claim~\ref{2Cases}, we have $\frak{m},+\infty\in\mathcal{R}_{f,g}^{\Omega}.$ 
So
\begin{equation}\label{+infty1}
[\frak{m},+\infty] \subset \mathcal{R}_{f, g}^\Omega.
\end{equation}
On the other hand, let $L\in \mathcal{R}_{f,g}^{\Omega}$ and $z^k\in \Omega$ be a sequence tending to $0$ such that $\displaystyle L=\lim_{k\to\infty}\frac{f(z^k)}{g(z^k)}$.
Set $s_k=g(z^k)\to 0.$ 
Then, for $k$ large enough, we have
$$
L=\displaystyle\lim_{k\to\infty}\frac{f(z^k)}{g(z^k)}
\geqslant\lim_{k\to\infty}\frac{1}{g(z^k)}\min_{z\in g^{-1}(g(z^k))\cap\Omega}f(z)
=\lim_{k\to\infty}\frac{1}{s_k}\min_{z\in g^{-1}(s_k)\cap\Omega}f(z)
=\frak{m}.
$$
Hence 
$$\mathcal{R}_{f, g}^\Omega \subset[\frak{m},+\infty].$$
Combining this with~\eqref{+infty1}, the desired conclusion follows.

iv) Let $\Omega_i,\, i\in I \ (\#I<+\infty),$ be the connected components of $\mathbb B_r\setminus g^{-1}(0).$
If $\#I=1,$ i.e., $\mathbb B_r\setminus g^{-1}(0)$ is connected, then the corollary holds immediately in view of Item~(i) so assume that $\#I>1.$
Before giving the proof, we need the following lemma.

\begin{lemma}\label{CB}
Let $J \subset I$ be a non-empty subset.
If $I \setminus J \ne \emptyset$, then there exist $i \in I \setminus J$ and $j \in J$ such that 
$$
0 \in \overline{\bigl(g^{-1}(0)\cap \overline{\Omega}_i \cap \overline{\Omega}_j\bigr)\setminus\{0\}} 
\quad \text{and} \quad 
\dim_0\bigl(g^{-1}(0)\cap \overline{\Omega}_i \cap \overline{\Omega}_j\bigr)=n-1.
$$
\end{lemma}

\begin{proof}
Let 
\begin{equation}\label{CD0}
\mathcal{F}=\{\Omega_i\}_{i\in I} \cup \{(\mathbb{B}_r \cap g^{-1}(0))\setminus\{0\}\} \cup \{\{0\}\}.
\end{equation}
By Proposition~\ref{CD1}, there exists a semi-algebraic stratification $\mathcal{S}=\{X_\alpha\}$ of $\mathbb{B}_r$ such that each member of $\mathcal{F}$ is a union of strata in $\mathcal{S}$.
For each integer $k>0$ with $\frac{1}{k}<r$, define
$$
U_k := \left(\bigcup_{i\in I \setminus J} \Omega_i\right) \cap \mathbb{B}_{\frac{1}{k}} \quad \text{ and }
\quad
U_k' := \left(\bigcup_{j\in J} \Omega_j\right) \cap \mathbb{B}_{\frac{1}{k}}.
$$
Choose
\begin{equation}\label{CD0-1}
a_k \in 
U_k \quad \text{ and }  
\quad
X := \bigcup_{\dim X_\alpha < n-1} X_\alpha.
\end{equation}
By Lemma~\ref{Join}, there exists $a_k' \in U_k'$ such that
\begin{equation}\label{CD1-0}
[a_k,a_k'] \cap X = \emptyset.
\end{equation}
Since $\{0\}$ is a stratum of $\mathcal{S}$ by~\eqref{CD0}, we have $0 \in X$, hence
\begin{equation}\label{CD1-1}
0 \notin [a_k,a_k'] \quad \text{for all } k.
\end{equation}
Define
$$
z^k = a_k + t_k^*(a_k' - a_k),
$$
where
$$
t_k^* := \max \left\{ t \in [0,1] : [a_k, a_k + t(a_k' - a_k)] \subset \overline{\bigcup_{i\in I \setminus J} \Omega_i} \right\}.
$$
Thus
\begin{equation}\label{CD2}
z^k \in \overline{\bigcup_{i\in I \setminus J} \Omega_i} 
= \bigcup_{i\in I \setminus J} \overline{\Omega}_i.
\end{equation}
Moreover,
$$
z^k \in \overline{\mathbb{B}_r \setminus \bigcup_{i\in I \setminus J} \overline{\Omega}_i}.
$$
Since $g \not\equiv 0$, the zero set $g^{-1}(0)$ has empty interior. Hence $\mathbb{B}_r \setminus g^{-1}(0)$ is dense in $\mathbb{B}_r$, and
$$
\mathbb{B}_r = \overline{\mathbb{B}_r \setminus g^{-1}(0)} 
= \overline{\bigcup_{i\in I} \Omega_i} 
= \bigcup_{i\in I} \overline{\Omega}_i.
$$
Combining with~\eqref{CD2}, we obtain
$$
z^k \in \bigcup_{j\in J} \overline{\Omega}_j.
$$
This, together with~\eqref{CD1-1} and~\eqref{CD2}, implies
$$
z^k \in \left[\left(\bigcup_{i\in I \setminus J} \overline{\Omega}_i\right) 
\cap \left(\bigcup_{j\in J} \overline{\Omega}_j\right)\right] \setminus \{0\}.
$$
Since $I \setminus J$ and $J$ are finite, by passing to a subsequence if necessary, there exist $i \in I \setminus J$ and $j \in J$ such that
\begin{equation}\label{CD3}
z^k \in (\overline{\Omega}_i \cap \overline{\Omega}_j)\setminus\{0\}.
\end{equation}
If $g(z^k) \ne 0$, then for $k$ large enough, $z^k \in \mathring{\mathbb{B}}_r \setminus g^{-1}(0)$, so $z^k$ lies in a unique connected component of $\mathbb{B}_r \setminus g^{-1}(0)$, contradicting~\eqref{CD3}. Hence
\begin{equation}\label{CD4}
z^k \in \bigl(g^{-1}(0)\cap \overline{\Omega}_i \cap \overline{\Omega}_j\bigr)\setminus\{0\}.
\end{equation}
Since $z^k \in [a_k,a_k'] \subset \mathbb{B}_{\frac{1}{k}}$, we have $\|z^k\| \le \frac{1}{k}$.
Thus
$$
0 \in \overline{\bigl(g^{-1}(0)\cap \overline{\Omega}_i \cap \overline{\Omega}_j\bigr)\setminus\{0\}}.
$$
By~\eqref{CD1-0}, we have $z^k \notin X$ for all $k$. 
Moreover, since $z^k \in g^{-1}(0)\setminus\{0\}$ and $(\mathbb{B}_r \cap g^{-1}(0))\setminus\{0\}$ is a union of strata, it follows that for each $k$ there exists a stratum $X_{\alpha_k} \subset g^{-1}(0)\setminus\{0\}$ of dimension $n-1$ such that $z^k \in X_{\alpha_k}$.
By the frontier condition,
$$
X_{\alpha_k} \subset \partial \Omega_i \cap \partial \Omega_j \subset \overline{\Omega}_i \cap \overline{\Omega}_j.
$$
Hence
$$
X_{\alpha_k} \subset \bigl(g^{-1}(0)\cap \overline{\Omega}_i \cap \overline{\Omega}_j\bigr)\setminus\{0\}.
$$
Therefore,
$$
\dim_0\bigl(g^{-1}(0)\cap \overline{\Omega}_i \cap \overline{\Omega}_j\bigr) = n-1,
$$
which completes the proof.
\end{proof}

Now let us prove Item~(iv). 
If $f$ changes sign on $g^{-1}(0)\cap \overline{\Omega}$ around $0$, where $\Omega$ is a connected component of $\mathbb{B}_r\setminus g^{-1}(0)$, then by Item~(ii), we have $\mathcal{R}_{f,g}^{\Omega}=\overline{\mathbb{R}}$. 
Consequently, $\mathcal{R}_{f,g}=\overline{\mathbb{R}}$, which is clearly connected. 
Therefore, we may assume that $f$ does not change sign on $g^{-1}(0)\cap \overline{\Omega}$  around $0$ for any connected component $\Omega$ of $\mathbb{B}_r\setminus g^{-1}(0)$.
Let us show that $f$ does not change sign on $g^{-1}(0)$ around $0.$
Let $i_1\in I.$
Since $I \setminus\{i_1\} \ne \emptyset,$ by Lemma~\ref{CB}, there is $i_2\in I \setminus\{i_1\}$ such that 
$$
0\in \overline{(g^{-1}(0)\cap \overline\Omega_{i_1}\cap \overline \Omega_{i_2}) \setminus\{0\}} \ \text{ and }\ \dim_0 (g^{-1}(0)\cap \overline\Omega_{i_1}\cap \overline \Omega_{i_2})=n-1.
$$
On the other hand, by assumption and by Lemma~\ref{CF}, 
$$\dim [f^{-1}(0) \cap g^{-1}(0)]<n-1.$$
Hence 
$$f(g^{-1}(0)\cap \overline\Omega_{i_1}\cap \overline \Omega_{i_2},0)\not\equiv 0.$$
This, together with the assumption that $f$ does not change sign on $g^{-1}(0)\cap \overline{\Omega}_{i_1}$ and $g^{-1}(0)\cap \overline{\Omega}_{i_2}$ around $0$, implies that $f$ does not change sign on $g^{-1}(0)\cap (\overline{\Omega}_{i_1}\cup \overline{\Omega}_{i_2})$ around $0.$
If $I \setminus\{i_1,i_2\} \ne \emptyset,$ we repeat the above argument to get $i_3\in I \setminus\{i_1,i_2\},$ such that $f$ does not change sign on $g^{-1}(0)\cap (\overline{\Omega}_{i_1}\cup \overline{\Omega}_{i_2}\cup \overline{\Omega}_{i_3})$ around $0.$
By induction, it can be shown that $f$ does not change sign on $\displaystyle g^{-1}(0)\cap \bigcup_{i\in I}\overline{\Omega}_{i}$ around $0.$
Hence, by Item~(iii), the set $\mathcal{R}_{f, g}^{\Omega_i}$ has only one of the following forms
$$[\frak{m}_i,+\infty],\ \ [\frak{M}_i,+\infty],\ \ [-\infty,\frak{M}_i]  \ \ \text{ or }\ \ [-\infty,\frak{m}_i] \ \ \text{for all}\ \  i\in I.$$ 
Therefore $\mathcal{R}_{f, g}$ has one of the following forms
$$\bigcup_{i\in I}[\frak{m}_i,+\infty],\ \ \bigcup_{i\in I}[\frak{M}_i,+\infty],\ \ \bigcup_{i\in I}[-\infty,\frak{M}_i]  \ \ \text{ or }\ \ \bigcup_{i\in I}[-\infty,\frak{m}_i],$$ 
which implies that it is connected.
\end{proof}

\section{The plane case}\label{2Case}
In this section, we study the numerical range in the case of rational function in two real variables. 
More precisely. we will show that this range can be computed explicitly by using Puiseux expansions.

Let $g \in \mathbb{R}[x, y]$ be a non-constant polynomial such that the origin is a non isolated zero of $g.$
Then for small $r > 0,$ the intersection $g^{-1}(0)$ with the punctured disk $\mathbb{B}_r \setminus \{(0, 0)\}$ has a fixed number of connected components, each is homeomorphic to a line and will be called a {\em half-branch} at the origin of $g^{-1}(0).$ 

\begin{remark}\label{OHB}
If $f$ and $g$ are coprime, by Lemma~\ref{CF}, we have 
$$\dim [f^{-1}(0) \cap g^{-1}(0)]<1,$$ 
i.e., $f^{-1}(0) \cap g^{-1}(0)$ is either empty or a finite set.
So for $r>0$ small enough each half-branch $\gamma$ of $g^{-1}(0)$, 
we have $f(z)\ne 0$ for any $z\in \gamma.$
In particular, $f$ is either positive or negative on $\gamma$.
\end{remark}

The following result is a version of Theorem~\ref{general-range} for the plane case.

\begin{theorem}\label{2-range}
Let $f$ and $g$ in $\mathbb{R}[x, y]$ be coprime polynomials with $f(0)=g(0)=0$ and $g\not\equiv 0$. 
Assume that $0$ is a non isolated zero of $g$ and let $\Omega$ be a connected component of $\mathbb{B}_r\setminus g^{-1}(0)$ whose boundary contains two half-branches 
$$\gamma=\gamma(\Omega) \ \text{ and }\ \gamma'=\gamma'(\Omega).$$
Then the following statements holds:
\begin{enumerate}[{\rm (i)}]
\item The set $\mathcal{R}_{f, g}^\Omega$ is a closed unbounded interval.
Moreover, if 
$$
\sign_0(f,\gamma\cup\gamma')\cap \sign(g,\Omega)\neq \emptyset,
$$
then $+\infty \in \mathcal{R}_{f,g}^{\Omega}$; and $-\infty \in \mathcal{R}_{f,g}^{\Omega}$ if 
$$
\sign_0(f,\gamma\cup\gamma')\cup \sign(g,\Omega)=\{\pm\}.
$$
\item If  $\sign(f,\gamma)$ and $\sign(f,\gamma')$ are opposite, then $\mathcal{R}_{f, g}^\Omega=\overline{\mathbb R}.$
\item  Assume that  $\sign(f,\gamma)$ and $\sign(f,\gamma')$ coincide. 
Let $F_{f,g}:=f_xg_y-f_yg_x$,
$$m=m(\Omega):=\left\{
\begin{array}{lll}
\displaystyle\min\left\{L:\ L\in \mathcal{R}_{f, g}^{F_{f,g}^{-1}(0)\cap\Omega}\right\}& \text{ if }\ 0\in \overline{ F_{f,g}^{-1}(0)\cap \Omega}\\
+\infty & \text{ otherwise}
\end{array}\right.
$$
and 
$$M=M(\Omega):=\left\{
\begin{array}{lll}
\displaystyle\max\left\{L:\ L\in \mathcal{R}_{f, g}^{F_{f,g}^{-1}(0)\cap\Omega}\right\} & \text{ if }\ 0\in \overline{ F_{f,g}^{-1}(0)\cap \Omega}\\
-\infty & \text{ otherwise.}
\end{array}\right.
$$
Then 
$$\mathcal{R}_{f, g}^\Omega=
\begin{cases}
[m,+\infty] & \text{if $\sign(f,\gamma\cup\gamma')= \sign(g,\Omega)$} \\
[-\infty,M] & \text{if $\sign(f,\gamma\cup\gamma')\ne\sign(g,\Omega)$.}
\end{cases}
$$
\end{enumerate}
\end{theorem}
\begin{proof}
(i) By Remark~\ref{OHB}, $f(\gamma\cup\gamma',0)\not\equiv 0$, so the statement follows from Theorem~\ref{general-range}(i).

(ii) Straightforwards from Theorem~\ref{general-range}(ii).

(iii) 
Consider the case $\sign(f,\gamma \cup \gamma') = \sign(g,\Omega) = \{+\}$.
The remaining cases can be treated similarly.
By Theorem~\ref{general-range}(iii), we have
$$
\mathcal{R}_{f,g}^{\Omega} = [\frak{m}, +\infty],
$$
where
$$
\frak{m} = \lim_{s \to 0} \frac{1}{s} \min_{z \in g^{-1}(s) \cap \Omega} f(z).
$$
Thus, it suffices to prove that $m = \frak{m}$.

Let $\alpha \colon (0,s_0) \to \Omega$ be a semi-algebraic arc defined in Claim~\ref{2Cases}, and set
$
\displaystyle\alpha_0 := \lim_{s \to 0} \alpha(s).
$
We distinguish two cases.
\subsubsection*{Case $1:\ \alpha_0 \ne 0$}
By Claim~\ref{2Cases}, we have $\frak{m} = +\infty$ and $\mathcal{R}_{f,g} = \{+\infty\}$.

If $0 \notin \overline{F_{f,g}^{-1}(0) \cap \Omega}$, then $m = +\infty$, hence $\frak{m} = m$.

If $0 \in \overline{F_{f,g}^{-1}(0) \cap \Omega}$, then $\mathcal{R}_{f,g}^{F_{f,g}^{-1}(0)\cap \Omega} \neq \emptyset$.
Since
$$
\mathcal{R}_{f,g}^{F_{f,g}^{-1}(0)\cap \Omega} \subset \mathcal{R}_{f,g}^{\Omega} = \{+\infty\},
$$
it follows that
$$
\mathcal{R}_{f,g}^{F_{f,g}^{-1}(0)\cap \Omega} = \{+\infty\}.
$$
Hence $m = +\infty = \frak{m}$.

\subsubsection*{Case $2:\ \alpha_0 = 0$}
Since $\alpha(s) \to 0$ as $s \to 0$, by shrinking $s_0$ if necessary, we may assume that
$\alpha(s) \notin \mathbb{S}_r$ for all $s \in (0,s_0)$.
By the Lagrange multiplier theorem, the vectors $\nabla f(\alpha(s))$ and $\nabla g(\alpha(s))$
are linearly dependent, hence
$$
F_{f,g}(\alpha(s)) \equiv 0.
$$
Therefore,
$$
\frak{m}
= \lim_{s \to 0} \frac{f(\alpha(s))}{g(\alpha(s))}
\in \mathcal{R}_{f,g}^{F_{f,g}^{-1}(0)\cap \Omega}
\subset \mathcal{R}_{f,g}^{\Omega} = [\frak{m}, +\infty],
$$
where the first equality follows from Claim~\ref{2Cases}.
This implies that $\frak{m} = m$.
\end{proof}

\begin{remark}
\begin{enumerate}[{\rm (i)}]
\item Note that under the setting of Theorem~\ref{2-range}, we cannot have $F_{f,g}\equiv 0$.
Indeed, suppose to the contrary that this holds. Then $f$ and $g$ are algebraically dependent (see, e.g.,~\cite[Theorem~2.2]{Ehrenborg1993}) (that is, there exists a nonzero polynomial $G\in \mathbb{R}[t_1,t_2]$ such that $G(f,g)\equiv 0$).
By~\cite[Exercise~10, p.~21]{Essen2000}, there exist polynomials $p,q\in \mathbb{R}[t]$ and $h\in \mathbb{R}[x,y]$ such that 
\[
f = p(h) \quad \text{and} \quad g = q(h).
\]
Since $f(0)=g(0)=0$, it follows that $p(h(0,0))=q(h(0,0))=0$, which implies that $h(0,0)$ is a common root of $p$ and $q$.
Hence $t - h(0,0)$ is a common factor of $p$ and $q$.
Therefore, $h(x,y) - h(0,0)$ is a common factor of $f$ and $g$.
Since $g\not\equiv 0$ and $g(0)=0$, the polynomial $g$ is nonconstant, and thus so is $h$.
This contradicts the assumption that $f$ and $g$ are coprime.

\item Theorem~\ref{2-range} still holds for $f$ and $g$ being continuous semi-algebraic functions.
\end{enumerate}
\end{remark}

As a consequence of Theorem~\ref{2-range}, we get immediately the following corollary.

\begin{corollary}\label{RangeFormula}
Set 
$$m^*:=\min_\Omega m(\Omega) \ \ \text{ and } \ \ M^*:=\max_\Omega M(\Omega).$$
Then the set $\mathcal{R}_{f,g}$ only has one of the following forms 
$$
[-\infty,+\infty], \quad [-\infty,M^*], \quad [m^*,+\infty], \quad \text{or} \quad [-\infty,M^*]\cup[m^*,+\infty],
$$
\end{corollary} 

Note that $m^*$ and $M^*$ can be compute by using Newton--Puiseux expansion.
However, to provide an algorithm that can be implemented in a computing software, we need to work with truncated Newton--Puiseux series.
To do this, we first need some preparation.

Let $f \in \mathbb{C}\{x,y\}$. If $f \not\equiv 0$, then $f$ admits a decomposition
$$
f = f_m + f_{m+1} + \cdots,
$$
where $f_m \not\equiv 0$ and, for each $k \ge m$, either $f_k \equiv 0$ or $f_k$ is a homogeneous polynomial of degree $k$. 
The {\em order} of $f$ is defined by
$$
\ord f := m.
$$
If $f \equiv 0$, we set $\ord f := +\infty$.

We say that $f$ is {\em $y$-regular of order $m$} if $\ord f = m$ and $f_m(0,1) \ne 0$. 
By a linear change of coordinates, any nonzero $f$ can be transformed into a $y$-regular polynomial of the same order.

\begin{lemma}\label{Lemma31}\cite[Lemma~3.1]{Dinh2025}
Let $c \in \mathbb{C}$ and define $\widetilde{f}(x,y) := f(x+cy,y)$. 
Then $\widetilde{f}$ is $y$-regular of order $m$ if and only if $f_m(c,1) \ne 0$.
\end{lemma}

Let $\gamma \not\equiv 0$ be a {\em Puiseux series} of the form
$$
\gamma(x) = c_1 x^{\frac{n_1}{n}} + c_2 x^{\frac{n_2}{n}} + \cdots,
$$
where $c_k \in \mathbb{C}$ and $n \le n_1 < n_2 < \cdots$ are positive integers with no common divisor, and such that the series $\gamma(t^n)$ has a positive radius of convergence. 
The {\em order} of $\gamma$ is defined by
$$
\ord \gamma := \frac{n_1}{n}.
$$
If $\gamma \equiv 0$, we set $\ord \gamma := +\infty$.

The series $\gamma$ is said to be {\em real} if all coefficients $c_k$ are real. 

Let  $$y=\gamma(x)=\sum_{\alpha} c_\alpha x^\alpha$$ be a Puiseux series. 
For each rational number $\rho$, the $\rho$-{\em approximation} of $\gamma$ is defined to be the series
$$y=\sum_{\alpha\leq \rho} c_\alpha x^\alpha.$$

Assume that $f\in \mathbb R[x,y]$ is $y$-regular.
A Puiseux series $\gamma$ is called a {\em Newton--Puiseux root} of $f$ if
$$
f(x,\gamma(x)) \equiv 0.
$$
The set of all {Newton--Puiseux} roots of $f$ is denoted by $\mathbf P(f)$.

 The {\em contact order} of a Puiseux series $\gamma$ with $f$ is defined as
$$\rho(\gamma,f):=\max\{\mathrm{ord}\left(\gamma-\gamma'\right):\ \gamma\neq \gamma'\in \mathbf P(f)\}.$$
If $\rho=\rho(\gamma,f)$, then the $\rho$-approximation of $\gamma$ is called {\em the $f$-truncation of $\gamma$}. 

If $\gamma$ is a {Newton--Puiseux} root of $f$, then we call the $f$-truncation of $\gamma$ a {\em truncated {Newton--Puiseux} root} of $f$, denoted by $\widetilde\gamma$. 
Denote by $\widetilde{\mathbf P}(f)$ the set of  truncated {Newton--Puiseux} roots $f$.

For each rational number $q$, a Puiseux series $\gamma$ is called a {\em Newton--Puiseux root$\mod q+$} of $f$ if there exists $\gamma'\in \mathbf P(f)$ such that 
$\mathrm{ord}\left(\gamma-\gamma'\right)>q.$
\begin{remark}\label{rm41}
{It follows from the definition that:}	
\begin{enumerate}[{\rm (i)}]
\item If $\gamma$ and $\gamma'$ are distinct {Newton--Puiseux} roots of $f$, then $\widetilde\gamma\neq \widetilde\gamma'$. That is, the natural map ${\mathbf P}(f)\to \widetilde{\mathbf P}(f)$ is bijective.
\item If $\gamma$ is a {Newton--Puiseux} root of $f$ then
$$\mathrm{ord}\left(\widetilde\gamma-\gamma\right)>\rho(\gamma,f)=\max\{\mathrm{ord}\left({\gamma}-\gamma'\right):\  \gamma\neq \gamma'\in \mathbf P(f)\}.$$
\end{enumerate}
\end{remark}

Assume that $f$ is $y$-regular, and let $C$ be a half-branch at the origin of the curve $f^{-1}(0)$ in $\mathbb{R}^2$. 
Then $C$ is not tangent to the $y$-axis. Consequently, along $C$, the first coordinate $x$ tends either to $0^+$ or to $0^-$. 
In the former case, we call $C$ a {\em right half-branch}, and in the latter case, a {\em left half-branch}.

Let
\begin{eqnarray*}
\mathbf{P}^{+}(f) &:=& \{\gamma \in \mathbf{P}(f) :\ \gamma \text{ is real}\}, \\
\mathbf{P}^{-}(f) &:=& \mathbf{P}^{+}(f(-x, y)).
\end{eqnarray*}
If $C$ is a right (resp., left) half-branch, then there exist $\delta > 0$ and a continuous semi-algebraic function 
$\gamma \colon [0,\delta) \to \mathbb{R}$ with $\gamma(0)=0$ such that $C$ is the germ of the curve
$$
(x,y) = (t,\gamma(t)) \quad \text{(resp., } (x,y)=(-t,\gamma(t))\text{)}
\quad \text{as } t \to 0^+.
$$
Moreover, $\gamma$ admits a Puiseux expansion satisfying $\gamma \in \mathbf{P}^{+}(f)$ (resp., $\gamma \in \mathbf{P}^{-}(f)$).
Conversely, each Puiseux series $\gamma \in \mathbf{P}^{+}(f)$ (resp., $\gamma \in \mathbf{P}^{-}(f)$) determines a right (resp., left) half-branch at the origin of the curve $f^{-1}(0).$

\begin{lemma}\label{Lemma32}
Let $y=\gamma(x)$ be a half-branch of $f\cdot g=0$ and let $\widetilde\gamma$ be its $(f\cdot g)$-truncation. Let $\rho$ be the contact order of $\widetilde\gamma$ and $f\cdot g$ and let $\widetilde\gamma_\tau:=\widetilde\gamma+\tau x^\rho\in \mathbb R(\tau)\{\!\{x\}\!\}$. Then $\widetilde\gamma$ is a {Newton--Puiseux} root$\mod \rho+$ of $g=0$ if and only if
$$\mathrm{ord}\ g\left(x,\widetilde\gamma(x)\right)>\mathrm{ord}\ g(x,\widetilde\gamma_\tau(x)).$$
\end{lemma}
\begin{proof} 
We write
$$g(x,y)=u(x,y)\cdot \prod_{\phi\in \mathbf P(g)} \left(y-\phi(x)\right)^{m_\phi}$$
with $u$ is a unit in $\mathbb C\{x,y\}$. Note that if $\phi\neq \gamma$ then
$$\mathrm{ord}\left(\widetilde\gamma-\phi\right)=\mathrm{ord}\left(\widetilde\gamma_\tau-\phi\right).$$
Then $\widetilde\gamma$ is a {Newton--Puiseux} root$\mod \rho+$ of $g=0$ if and only if
$$\mathrm{ord}\left(\widetilde\gamma-\phi\right)>\rho,\quad\text{for some}\ 
\phi\in\mathbf{P}(g),$$ which is equivalent to
$$\mathrm{ord}\ g\left(x,\widetilde\gamma(x)\right)>\mathrm{ord}\ g(x,\widetilde\gamma_\tau(x)).$$
\end{proof}

\begin{lemma}\label{Lemma34}
Let $y=\gamma(x)$ be a half-branch of $g=0$ and let $\widetilde\gamma$ be its $(f\cdot g)$-truncation. 
Then the sign of $f$ along $\gamma$ is equal to the sign of the leading coefficient of $f(x,\widetilde\gamma(x))$ if $\gamma$ is a right half-branch and is equal to the sign of the leading coefficient of $f(-x,\widetilde\gamma(x))$ if $\gamma$ is a left half-branch.
\end{lemma}
\begin{proof} 
It suffices consider the case $\gamma$ is a right half branch and  to show that $\mathrm{LT}\left(f(x,\gamma(x))\right)=\mathrm{LT}\left(f(x,\widetilde\gamma(x))\right)$ where $\mathrm{LT}(\cdot)$ denotes the leading term. 
In fact, write
$$f(x,y)=u\cdot \prod_{\phi\in \mathbf P(f)} \left(y-\phi(x)\right)^{m_\phi}$$
with $u$ is a unit in $\mathbb C\{x,y\}$. Note that for all $\phi\in \mathbf P(f)\subset \mathbf P(f\cdot g)$, one has
$$\mathrm{ord}\left(\gamma-\phi\right)\leq \rho(\gamma,f\cdot g)<\mathrm{ord}\left(\widetilde\gamma-\gamma\right).$$
Since
$$\widetilde\gamma(x)-\phi(x)=\widetilde\gamma(x)-\gamma(x)+\gamma(x)-\phi(x)$$
it follows that
$$\mathrm{LT}\left(\gamma(x)-\phi(x)\right)=\mathrm{LT}\left(\widetilde\gamma(x)-\phi(x)\right).$$
Hence $\mathrm{LT}\left(f(x,\gamma(x))\right)=\mathrm{LT}\left(f(x,\widetilde\gamma(x))\right)$. This completes the lemma.
\end{proof}

In the following, for two right/left half-branches $\gamma$ and $\gamma'$, we write $\gamma>\gamma'$ if $\gamma(x)>\gamma'(x)$ for all $x>0$ sufficiently small.

\begin{lemma}\label{Lemma33} 
Let $\Omega$ be a connected component of $\mathbb B_r\setminus g^{-1}(0)$ whose boundary contains two half-branches $y=\gamma(x)$ and $ y=\gamma'(x)$. 
Let  $y=\widetilde\gamma(x)$ and $y=\widetilde{\gamma}'(x)$ be the $(f\cdot g)$-truncations of  $\gamma$ and $\gamma'$ respectively. 
Then the following statements hold:
\begin{enumerate}[{\rm (i)}]
\item If both $\gamma$ and $\gamma'$ are right half-branches, then the half-branch
$$\beta_1:=\left\{
\begin{array}{lll}
\left\{\left(x,\displaystyle\frac{\widetilde\gamma(x)+\widetilde{\gamma}'(x)}{2}\right): \ 0<x\ll 1\right\}&\text{if $\Omega \subset \{x>0\}$}\\
\{(-x,0): \ 0< x\ll 1\} &\text{otherwise.}   
\end{array}\right.
$$
belongs to $\Omega$.
\item If both $\gamma$ and $\gamma'$ are left half-branches, then the half-branch
$$\beta_2:=\left\{
\begin{array}{lll}
\left\{\left(-x,\displaystyle\frac{\widetilde\gamma(x)+\widetilde{\gamma}'(x)}{2}\right): \ 0<x\ll 1\right\}&\text{if $\Omega \subset \{x<0\}$}\\
\{(x,0): \ 0<x\ll 1\} &\text{otherwise.}   
\end{array}\right.
$$
belongs to $\Omega$.
\item If one of $\gamma$ and $\gamma'$ is a right half-branch and the other is a left half-branch, then one of  the branches 
$$\beta_3:=\{(0,y):\ 0<y\ll 1\} \ \text{ and }\ \beta_4:=\{(0,-y):\ 0<y\ll 1\}$$ 
is contained in $\Omega$.
\end{enumerate}
Hence, in all cases, the sign of $g$ in $\Omega$ is the sign of the leading coefficient of $g(\beta_i)$ $(i=1,\dots,4)$.
\end{lemma}
\begin{proof} 
Item~(iii) and the last statement are trivial. 
Moreover, we only need to prove Item~(i), since the the proof of Item~(ii) is completely similar. 
We may assume that $\gamma(x)< \gamma'(x)$. 

Assume first that $\Omega \subset \{x>0\}$, so $\Omega$ is the region in the right half-plane lying between $\gamma$ and $\gamma'$.
Write
$$\gamma'(x)- \gamma(x)=cx^\alpha+o(x^\alpha) \ \text{ with }\ c>0.$$
Then 
$$\alpha\leq \rho(\gamma',g)<\mathrm{ord}\left(\widetilde\gamma'-\gamma'\right)\text{ and }\alpha\leq \rho(\gamma,g)<\mathrm{ord}\left(\widetilde\gamma-\gamma\right).$$
Letting $\varphi(x)=\frac{\widetilde\gamma(x)+\widetilde{\gamma}'(x)}{2}$, we have
\begin{align*}
2(\varphi(x)-\gamma(x))&=\left(\widetilde\gamma(x)-\gamma(x)\right)+\left(\widetilde\gamma'(x)-\gamma(x)\right)\\
&=\widetilde\gamma'(x)-\gamma(x)+o(x^\alpha)\\
&=\widetilde\gamma'(x)-\gamma'(x)+\gamma'(x)-\gamma(x)+o(x^\alpha)\\
&=cx^\alpha+o(x^\alpha)>0.
\end{align*}
This means that $\varphi(x)>\gamma(x)$. Analogously, one has $\varphi(x)<\gamma'(x)$ and hence the half-branch $y=\varphi(x)$ belongs to $\Omega$.

Now suppose that $\Omega \not\subset \{x>0\}$.
Then the intersection of $\Omega$ with the half-plane $\{x>0\}$ contains the region lying below $\gamma$ and above $\gamma'$.
In particular, $\mathbb{B}_r \cap \{x<0\} \subset \Omega$.
Hence
\[
\{(-x,0):\  0<x\ll 1\} \subset \Omega.
\]
The lemma is proved.
\end{proof}

\begin{lemma}\label{Lemma35}
Let $\Omega$ be a connected component of $\mathbb B_r\setminus g^{-1}(0)$ with boundary containing two half-branches $y=\gamma(x)$ and $y=\gamma'(x)$. 
Let  $y=\widetilde\gamma(x)$ and $y=\widetilde{\gamma'}(x)$ be the $(f\cdot g\cdot F_{f,g})$-truncations of  $\gamma,\gamma'$. Let $y=\varphi(x)$ be a Puiseux root of $F_{f,g}$ and let $\widetilde\varphi(x)$ be its $(f\cdot g\cdot F_{f,g})$-truncation. Then the following assertions are equivalent
\begin{itemize}
\item[(i)] $\varphi(x)\in (\gamma(x),\gamma'(x))$, i.e. $y=\varphi(x)\in\Omega$.
\item[(ii)] $\widetilde\varphi(x)\in (\widetilde\gamma(x),\widetilde\gamma'(x))$.
\end{itemize}
Moreover, if (i) holds then
\begin{align}\label{nochangelimit}
\lim_{x\to 0} \frac{f\left(x,\widetilde\varphi(x)\right)}{g\left(x,\widetilde\varphi(x)\right)}=\lim_{x\to 0} \frac{f\left(x,\varphi(x)\right)}{g\left(x,\varphi(x)\right)}.
\end{align}
\end{lemma}
\begin{proof} 
(i)$\Rightarrow$(ii): Assume that $\varphi(x)\in (\gamma(x),\gamma'(x))$.  Since
$$\mathrm{ord}\left(\widetilde\varphi-\varphi \right)>\mathrm{ord}\left(\gamma-\varphi \right)\text{ and }\mathrm{ord}\left(\widetilde\gamma-\gamma\right)>\mathrm{ord}\left(\gamma-\varphi \right),$$
it follows that
\begin{align*}
\mathrm{LT}\left(\widetilde\gamma(x)-\widetilde\varphi(x)\right)&=\mathrm{LT}\left(\left(\widetilde\gamma(x)-\gamma(x)\right)+\left(\gamma(x)-\varphi(x)\right)+\left(\varphi(x)-\widetilde\varphi(x)\right)\right)\\
&=\mathrm{LT}\left(\gamma(x)-\varphi(x)\right).
\end{align*}
This implies that $\widetilde\gamma(x)<\widetilde\varphi(x)$, and hence $\widetilde\varphi(x)\in (\widetilde\gamma(x),\widetilde\gamma'(x))$.

Notice that the correspondence $\gamma\mapsto \widetilde\gamma$ is a bijection. Combining with the implication (i)$\Rightarrow$(ii) we obtain the inverse (ii)$\Rightarrow$(i).

To prove the identity (\ref{nochangelimit}) we write
$$g(x,y)=u(x,y)\cdot \prod_{\gamma\in \mathbf P(g)} \left(y-\gamma(x)\right)^{m_\gamma}$$
with $u$ is a unit in $\mathbb C\{x,y\}$. Notice that, if $\varphi\neq \gamma$ then
$$\mathrm{ord}\left(\widetilde\varphi-\varphi \right)>\rho(\varphi,f\cdot g\cdot F)\geq\mathrm{ord}\left(\varphi-\gamma\right).$$
It yields that
\begin{align*}
\mathrm{LT}\left(\widetilde\varphi-\gamma\right)&=\mathrm{LT}\left(\left(\widetilde\varphi-\varphi\right)+\left(\varphi-\gamma)\right)\right)\\
&=\mathrm{LT}\left(\varphi-\gamma\right)
\end{align*}
and therefore 
$$\mathrm{LT}\left(g\left(x,\widetilde\varphi(x)\right)\right)=\mathrm{LT}\left(g\left(x,\varphi(x)\right)\right).$$
Similarly, one has
$$\mathrm{LT}\left(f\left(x,\widetilde\varphi(x)\right)\right)=\mathrm{LT}\left(f\left(x,\varphi(x)\right)\right)$$
and \eqref{nochangelimit} follows.
\end{proof}

\section{Algorithm for computing numerical range}\label{Algorithm}

INPUT: Two polynomials $f$ and $g$ in $\mathbb{Q}[x,y]$ of positive orders. 

OUTPUT: Decide whether or not the limit $\displaystyle\lim_{(x, y) \to (0, 0)}\frac{f(x, y)}{g(x, y)}$ exists and compute the limit/range.
\footnote{We implemented the algorithm in a computer algebra system. The code for Maple, together with a Maple worksheet, is available at the following links: \url{https://drive.google.com/file/d/1jdEIYZ_55YrN60EKTu3EjyMsWDY7y5O5/view?usp=sharing}.}
\begin{enumerate}[{\rm Step 1.}]


\item Compute $p \in \mathbb{Q}[x, y],$ which is the greatest common
	divisor of the polynomials $f$ and $g.$ Replace $f$ and $g$ by the
	new polynomials $\frac{f}{p}$ and $\frac{g}{p},$ respectively. If
	$g(0, 0) \ne 0,$ then the limit $\displaystyle\lim_{(x, y) \to (0, 0)}\frac{f(x, y)}{g(x, y)}$ exists and equals to $\displaystyle\frac{f(0,0)}{g(0,0)}$  and the algorithm stops. Otherwise, proceed to the next step.


\item Set $\displaystyle F_{f,g}:= \frac{\partial f}{\partial x}\frac{\partial g}{\partial y}-\frac{\partial f}{\partial y}\frac{\partial g}{\partial x}$. 
If one of the polynomials $f, g$ and $F_{f, g}$ is not $y$-regular, make a linear transformation, so that the new polynomials $f, g,$ and $F_{f,g}$ are $y$-regular.


\item Let $M$ be the smallest integer\footnote{Maple can determine the best order of truncation to distinguish Newton--Puiseux roots of $g$ by using the command \textsf{puiseux}$(g,x=0,y,0)$.}
such that, for $*\in\{\pm\}$, the truncation mapping $\gamma \mapsto \gamma \mod x^M$ induces a bijection between 
$\mathbf{P}^{*} \left(g\right)$ and $\mathbf{P}_M^{*} \left(g\right)$
and compute $\mathbf{P}^{*}_M({g}).$ 
If $\mathbf{P}^{+}_M({g})\cup\mathbf{P}^{-}_M({g})$ is empty, then proceed with the algorithm BiLimit.
Otherwise, proceed to the next step.

\item Compute $\mathbf{P}^{*}_N(f\cdot g\cdot{F}_{f, g})$ where $N\geqslant M$ is the smallest integer such that, for $*\in\{\pm\}$, the truncation mapping $\gamma\mapsto \gamma \mod x^N$ induces bijections between $\mathbf{P}^{*} \left(f\cdot g\cdot F_{f,g}\right)$ and $\mathbf{P}_N^{*} \left(f\cdot g\cdot F_{f,g}\right)$.

\item For $*\in\{\pm\}$, arrange the elements of  $\mathbf{P}^{*}_N(f\cdot g\cdot{F}_{f, g})$ as follows:
$$\mathbf{P}^{+}_N(f\cdot g\cdot{F}_{f, g})=\{\widetilde\gamma_1>\cdots >\widetilde\gamma_{k^+}\},$$
$$\mathbf{P}^{-}_N(f\cdot g\cdot{F}_{f, g})=\{\widetilde\gamma_{k^++1}<\cdots <\widetilde\gamma_k\},$$
where $k^+:=0$ if $\mathbf{P}^{+}_N(f\cdot g\cdot{F}_{f, g})=\emptyset.$ 
Then on 
$$\mathbf{P}_N(f\cdot g\cdot{F}_{f, g})=\mathbf{P}^{+}_N(f\cdot g\cdot{F}_{f, g})\cup \mathbf{P}^{-}_N(f\cdot g\cdot{F}_{f, g}),$$
we may define the following cyclic order  
$$\widetilde\gamma_1>_c\cdots >_c\widetilde\gamma_{k^+}>_c\widetilde\gamma_{k^++1}>_c\cdots >_c\widetilde\gamma_k>_c\widetilde\gamma_1.$$

\item In view of Lemma~\ref{Lemma32}, the truncated half-branches which are truncated Newton Puiseux roots of $g$ are given by
$$
\begin{array}{lrl}
\mathbf{P}_N(g)&:=&\{\widetilde\gamma\in \mathbf{P}_N(f\cdot g\cdot{F}_{f, g}):\  \mathrm{ord}\ g\left(x,\widetilde\gamma(x)\right)> \mathrm{ord}\ g(x,\widetilde\gamma(x)+\tau x^N)\}\\
&=&\{\widetilde\gamma_{k_1}>_c\cdots>_c \widetilde\gamma_{k_m}>_c \widetilde\gamma_{k_1} \}.
\end{array}
$$
Let $k_{m+1}:=k_m+1$ and $\widetilde\gamma_{k_{m+1}}:=\widetilde\gamma_{k_1}$. 
If $m\ne 2$, for $i=1,\dots,m,$ let $\Omega_i$ be the connected component of $\mathbb{B}_r\setminus g^{-1}(0)$ whose boundary contains the branches given by $\gamma_{k_i}$ and $\gamma_{k_{i+1}}.$
When $m=2,$ the set $\mathbb{B}_r\setminus g^{-1}(0)$ has two connected components whose boundary contains $\gamma_{k_1}$ and $\gamma_{k_2}.$
So to avoid any confusion, set $\Omega_{k_1}$ to be the connected component containing points obtained by rotating  $\gamma_{k_1}$ clockwise to $\gamma_{k_2}.$

\item By Lemma~\ref{Lemma34}, compute the sign of $f$ on each half-branch of $\mathbf{P}_N(g)$ as follows: 
\begin{itemize}
\item If $k_i \leqslant k^+,$ set
$$\sigma_i:=\left\{
\begin{array}{lll}
+ &\text{if $\mathrm{LT}\left(f(x,\widetilde\gamma_{k_i}(x))\right)>0$}\\
- &\text{if $\mathrm{LT}\left(f(x,\widetilde\gamma_{k_i}(x))\right)<0$}.
\end{array}\right.
$$
\item If $k_i > k^+,$ set
$$\sigma_i:=\left\{
\begin{array}{lll}
+ &\text{if $\mathrm{LT}\left(f(-x,\widetilde\gamma_{k_i}(x))\right)>0$}\\
- &\text{if $\mathrm{LT}\left(f(-x,\widetilde\gamma_{k_i}(x))\right)<0$}.
\end{array}\right.
$$
\end{itemize}

If there exists $i\in\{1,\dots,m-1\}$ such that $\sigma_i\ne\sigma_{i+1}$ then conclude that $\mathcal{R}_{f,g}=\overline{\mathbb{R}}$ by Theorem \ref{2-range} (ii) and stop the algorithm.
Otherwise, proceed to the next step.
 
\item For $i=1,\dots,m$, compute an arc lying in $\Omega_i$ given by Lemma~\ref{Lemma33}:  
$$
\beta_i=\left\{
\begin{array}{cll}
\displaystyle\left(x,\frac{\widetilde\gamma_{k_i}(x)+\widetilde\gamma_{k_{i+1}}(x)}{2}\right)& \text{if $1 \leqslant k_i< k_{i+1} \leqslant k^+$ }\\
\displaystyle\left(0,-y\right)& \text{if $1 \leqslant k_i\leqslant k^+< k_{i+1}\leqslant k_m $ }\\
\displaystyle\left(-x,0\right)& \text{if $1 <k_i\leqslant k^+< k_{i+1}= k_{m+1} $ }\\
\displaystyle\left(-x,\frac{\widetilde\gamma_{k_i}(x)+\widetilde\gamma_{k_{i+1}}(x)}{2}\right)& \text{if $k^+< k_i< k_{i+1} \leqslant k_m$ }\\
\displaystyle\left(0,y\right)& \text{if $0<k^+< k_i=k_m< k_{i+1}=k_{m+1}$ }\\
\displaystyle\left(x,0\right)& \text{if $0=k^+< k_i=k_m< k_{i+1}=k_{m+1}$ }.

\end{array}\right.
$$

\item By Lemma~\ref{Lemma34}, for $i=1,\dots,m$, the sign of $g$ on each $\Omega_i$ con be compute by it sign on the half-branch $\beta_i$ 
$$\widetilde\sigma_i:=\left\{
\begin{array}{lll}
+ &\text{if $\mathrm{LT}\left(g(\beta_i)\right)>0$}\\
- &\text{if $\mathrm{LT}\left(g(\beta_i)\right)<0$}.
\end{array}\right.
$$

\item By Lemma~\ref{Lemma35}, for $i=1,\dots,m$, the truncations of Newton Puiseux roots of $F_{f,g}$ contained in $\Omega_i$ are given by
$$\mathbf{P}_N^{\Omega_i}(F_{f,g}):=\left\{
\begin{array}{lrl}
\widetilde\gamma_j\in \mathbf{P}_N(f\cdot g\cdot{F}_{f, g})&:&  \mathrm{ord}\ F_{f,g}\left(x,\widetilde\gamma_j(x)\right)> \mathrm{ord}\ F_{f,g}(x,\widetilde\gamma_j(x)+\tau x^N),\\
&&k_i<j<k_{i+1}
\end{array}\right\}.
$$
Moreover, let
$$\mathbf{P}_{1}:=\{i=1,\dots,m:\ \sigma_i=\widetilde\sigma_i\},$$
$$\mathbf{P}_{2}:=\{i=1,\dots,m:\ \sigma_i\ne\widetilde\sigma_i\},$$

\item Compute
$$m^*:=\min\left\{\min_{i\in P_1}\min_{\widetilde\gamma\in \mathbf{P}_N^{\Omega_i}(F_{f,g})}\lim_{x\to 0^+}\frac{f(x,\widetilde\gamma(x))}{g(x,\widetilde\gamma(x))};\ +\infty\right\},$$
$$M^*:=\max\left\{\max_{i\in P_2}\max_{\widetilde\gamma\in \mathbf{P}_N^{\Omega_i}(F_{f,g})}\lim_{x\to 0^+}\frac{f(x,\widetilde\gamma(x))}{g(x,\widetilde\gamma(x))};\ -\infty\right\}.$$


\item The range of the limit $\displaystyle\lim_{(x,y)\to (0,0)}\frac{f(x,y)}{g(x,y)}$ is
$$\mathcal{R}_{f,g}=\left\{
\begin{array}{lll}
\text{$[-\infty,M^*]\cup [m^*,+\infty]$} & \text{if $\mathbf{P}_1\ne \emptyset$ , $\mathbf{P}_2\ne \emptyset$ and  $M^*<m^*$}\\
\text{$[-\infty,M^*]$} & \text{if $\mathbf{P}_1=\emptyset$ and $\mathbf{P}_2\ne \emptyset$}\\
\text{$[m^*,+\infty]$} & \text{if $\mathbf{P}_1\ne\emptyset$ and $\mathbf{P}_2= \emptyset$}\\
\text{$[-\infty,+\infty]$} & \text{otherwise.}
\end{array}\right.
$$


\end{enumerate}

\subsubsection*{Acknowledgements.} The authors would like to express their gratitude to the Vietnam Institute for Advanced Study in Mathematics (VIASM) for their hospitality and support during the writing of this paper.

\end{document}